\documentclass[11pt,twoside]{article}
\usepackage{fancyhdr,color,graphicx,amsmath,amsfonts,amssymb,latexsym,bm,indentfirst,cite,amsthm,enumerate}
\usepackage[colorlinks=true,backref=page]{hyperref}
\usepackage[all]{xy}
\usepackage[a4paper,left=25mm,right=25mm,top=35mm,body={155mm,230mm}]{geometry}
\usepackage{pdfpages}
\usepackage{titletoc}
\usepackage{fancyhdr}
\newtheorem*{claim}{\hspace{2em} Claim}

\newtheorem{Theorem}{\hspace{2em}Theorem}[section]%按章排序号
\newtheorem{Definition}[Theorem]{\hspace{2em} Definition}
\newtheorem{Lemma}[Theorem]{\hspace{2em} Lemma}

\newtheorem{Proposition}[Theorem]{\hspace{2em} Proposition}
\newtheorem*{theoremm}{Main Theorem}

\makeatletter
\newcommand{\subsectionruninhead}{\@startsection{subsection}{2}{0mm}{-\baselineskip}{-0mm}{\bf\large}}
\newcommand{\subsubsectionruninhead}{\@startsection{subsubsection}{3}{0mm}{-\baselineskip}{-0mm}{\bf\normalsize}}
\makeatother
\begin{document}
	\newpage
	\title{Approximation property on entropies for surface diffeomorphisms}
	\author{Wanlou Wu \and Jiansong Liu}
	\date{}
	\maketitle

\begin{abstract}
	In this paper, we prove that for any $C^1$ surface diffeomorphism $f$ with positive topological entropy, there exists a diffeomorphism $g$ arbitrarily close (in the $C^1$ topology) to $f$ exhibiting a horseshoe $\Lambda$, such that the topological entropy of $g$ restricted on $\Lambda$ can arbitrarily approximate the topological entropy of $f$. This extends the Theorem \cite[Theorem 1.1]{Gan} of Gan. 
	%We describe the entropy of the diffeomorphism $f$ and $g$ in quantity. 
\end{abstract}

\section{Introduction}
	
    The entropy of a dynamical system is a non-negative real number which measures the complexity of the system. Modeled on the definition of the Kolmogorov-Sinai \cite{K1,S}, or metric entropy, Adler, Konheim and McAndrew \cite{AKM} first introduced the concept of topological entropy in 1965. Later, Rufus Bowen \cite[Definition 2]{BR} and Dinaburg \cite{DEI} gave a different, weaker definition of topological entropy imitating that of the Hausdorff dimension.
	
	In 1960s, Smale \cite{SS} found the horseshoe and created the hyperbolic theory. He proved that if a dynamical system has a hyperbolic periodic point with transversal homoclinic point, then there exists an invariant hyperbolic set such that the dynamics on it is topologically conjugate to a topological Markov chain. It means that the topological entropy of the original system is positive. In \cite[Corollary 4.3]{Ka}, Katok confirmed that if $f$ is a $C^{1+\alpha}(\alpha>0)$ surface diffeomorphism with positive topological entropy, then the diffeomorphism $f$ has a hyperbolic periodic point with a transversal homoclinic point. Gan \cite[Theorem 1.1]{Gan} proved that for any $C^1$ surface diffeomorphism $f$ with positive topological entropy, there exists a diffeomorphism $g$ which can arbitrarily close $f$ in the $C^1$ topology exhibiting a transversal homoclinic point. These papers describe the relationship between entropy and horseshoe in quality. It is natural to ask the following question: can we give some description about the relationship between entropy and horseshoe in quantity? Especially, in Gan's paper \cite{Gan}, does the entropy of the diffeomorphism $g$ and $f$ equal? Or what is the difference between the entropy of the diffeomorphism $g$ and $f$? In this paper, we give a partial answer to this question.     
	
	\bigskip
	
\begin{theoremm}\label{Thm:main}
    Let $f$ be a $C^1$ diffeomorphism on a compact two-dimensional $C^\infty$ Riemannian manifold $M$ without boundary. If $f$ has positive topological entropy, then for any $\varepsilon>0$ and any $C^1$ neighborhood $\mathcal{U}$ of $f$, there exists a diffeomorphism $g\in\mathcal{U}$ exhibiting a horseshoe $\Lambda$ such that $$\lvert h_{top}(g|_{\Lambda})-h_{top}(f)\rvert<\varepsilon,$$ where $h_{top}(f)$ is the topological entropy of $f$.
\end{theoremm}

	We introduce the outline of proof on the Main Theorem:
	
	According to the variational principle (Theorem \ref{VP}) and Ruelle's inequality (Theorem \ref{RI}), one can take an hyperbolic ergodic measure $\mu$ with positive metric entropy. For $\mu$-almost every point $p\in M$, by Ma\~n\'e's ergodic closing lemma (Lemma \ref{TECL}), there exists a sequence of diffeomorphisms $\{f_n\}$ and a sequence of points $\{p_n: p_n \text{ is the periodic point of } f_n\}$ such that $$\lim_{n\rightarrow\infty}f_n=f {\rm \ and \ }\lim_{n\rightarrow\infty}p_n=p.$$  Abdenur, Bonatti, and Crovisier \cite[Proposition 6.1]{ABC} proved that the Dirac measures along these periodic orbits converge weakly to $\mu$ and whose Lyapunov exponents converge to the Lyapunov exponents of $\mu$. For a periodic point $z$, let $$\mathcal{O}rb(f,z)\triangleq\{p,f(z),\cdots,f^{\pi(z)-1}(z)\},~\text{ where }\pi(z)\triangleq min\{n\in\mathbb{N}^+: f^{n}(z)=z\}.$$ We discuss two different cases.
	
\begin{itemize} 
	\item[---] If there is no uniformly dominated splitting over the set $\mathcal{O}rb(f_n, p_n)$ for all $n$, then we can assume that there exists no dominated splittings over $\mathcal{O}rb(f_n, p_n)$ for all $n$. According to Buzzi, Crovisier and Fisher \cite[Theorem 4.1]{BCF}, there exists a diffeomorphism $g$ containing a horseshoe $\Lambda$ with big topological entropy. Then, we can get the main result by the property of intermediate entropy of horseshoes.
	
    \item[---] If there exists a uniformly dominated splitting over $\mathcal{O}rb(f_n, p_n)$ for all $n$, then there is a dominated splitting over $supp(\mu)$. By applying a result from Gelfert\cite[Theorem 1]{GEL}, we show the quantitative relationship between the topological entropy and horseshoe.
\end{itemize}

\section{Entropy and Lyapunov exponents}
	
    Assume that $(M, \mathcal{A})$ is a measurable space, $f:M\rightarrow M$ is a measurable transformation preserving a probability measure. The set of all invariant measures and all ergodic measures are denoted by $\mathcal{M}(f)$ and $\mathcal{E}(f)$ respectively. Let $\xi=\{A_1,A_2,\cdots,A_k\},~\eta=\{B_1,B_2,\cdots,B_l\}$ be two finite measurable partitions of $M$, the union $\xi\bigvee\eta$ of $\xi$ and $\eta$ is defined as $$\xi \bigvee \eta=\{A_i \bigcap B_j: 1\leq i \leq k,~1\leq j \leq l\}.$$ The union $\xi\bigvee\eta$ is also a measurable partition of $M$. Given a natural number $n\geq 1$, $$\bigvee_{i=0}^{n-1}f^{-i}(\xi)\triangleq \xi\bigvee f^{-1}(\xi)\bigvee\cdots\bigvee f^{-n+1}(\xi).$$  
	
\begin{Definition}
	Given a measure $\mu\in\mathcal{M}(f)$ and a finite measurable partition $\xi=\{A_1,A_2,\cdots,A_k\}$ of $M$, the metric entropy of the measurable partition $\xi$ is defined as $$H_\mu(\xi)=-\sum_{i=1}^{k}\mu(A_i)\log\mu(A_i).$$ The metric entropy of transformation $f$ w.r.t. $\xi$ is defined as $$h_\mu(f,\xi)=\lim_{n\rightarrow\infty}\frac{1}{n}\log H_\mu(\bigvee_{i=0}^{n-1}f^{-i}(\xi)).$$ The metric entropy of transformation $f$ w.r.t. $\mu$ is defined as $$h_{\mu}(f)\triangleq\sup \{h_{\mu}(f,\xi): \xi \text{ is a finite measurable partition of M}\}.$$
\end{Definition}

\begin{Definition}
    Given an open cover $\xi$ of $M$, let $N(\bigvee_{i=0}^{n-1}f^{-i}(\xi))$ be the minimal cardinality of a subcover of $\bigvee_{i=0}^{n-1}f^{-i}(\xi)$. The entropy of transformation $f$ w.r.t. $\xi$ is defined as $$h(f,\xi)=\lim_{n\rightarrow\infty}\dfrac{1}{n}\log N(\bigvee_{i=0}^{n-1}f^{-i}(\xi)).$$ The topological entropy of transformation $f$ is defined as $$h_{top}(f)\triangleq \sup\{h(f,\xi): \xi \text{ is an open cover of $M$}\}.$$ 
\end{Definition}
	
	Topological entropy was actually discovered later than metric entropy. Metric entropy gives a quantitative measure of the complexity of a dynamical system as seen via an invariant measure. Topological entropy was found by extracting from the same concept an invariant of the topological dynamics only. The topological entropy measures the maximal dynamical complexity versus an average complexity reflected by metric entropy. Therefore, metric entropy is not greater than topological entropy and measures assigning most weight to regions of high complexity should have metric entropy close to the topological entropy. This is the famous variational principle.
	   
\begin{Theorem}\label{VP}
	(The Variational Principle, \cite[Theorem 8.6]{PW}) If $f$ is a continuous map on a compact metric space $M$, then $$ h_{top}(f)=\sup\{h_{\mu}(f): \mu\in\mathcal{M}(f)\}.$$ By the relationship between invariant measures and ergodic measures, one has that $$ h_{top}(f)=\sup\{h_{\mu}(f): \mu\in\mathcal{E}(f)\}.$$
\end{Theorem}  
	
    From now on, let $M^d$ be a $d$-dimensional compact Riemannian manifold and $f$ a diffeomorphism on $M^d$. We denote the tangent map of diffeomorphism $f$ at the point $x\in M^d$ by $Df_x$. Given a measure $\mu\in \mathcal{M}(f)$, Oseledec's theorem \cite[Theorem 3]{OV} affirms that for $\mu$-almost every $x\in M$, there exist real numbers $\lambda_{1}(x)<\lambda_{2}(x)<\cdots<\lambda_{s(x)}(x)$ and a splitting $T_{x}M=E_{1}(x)\oplus E_{2}(x)\oplus\cdots \oplus E_{s(x)}(x)$, where $s(x)\in(1,d]$ is a measurable function and  $s(x)\in\mathbb{N}$, such that $$Df_x(E_i(x))=E_i(f(x)), {\rm and }~ \lambda_i(x)=\lim_{n\rightarrow \pm\infty} \frac{1}{n} \log \parallel Df^{n}_{x}v\parallel,~{\rm for \ any} ~ v\in E_{i}(x)\setminus\{0\},~i\in[1,s(x)].$$ The numbers $\lambda_i(x)$ are called \emph{Lyapunov exponents} of $f$ at the point $x$ and the numbers $m_i(x)\triangleq$ dim$E_{i}(x)$ are called the \emph{multiplicity} of $\lambda_i(x)$. Especially, if the measure $\mu\in\mathcal{E}(f)$, then for $\mu$-almost every $x\in M$, one can get uniform constant $s=s(x)$ and uniform exponents $\lambda_i(\mu)=\lambda_i(x)$, $m_i(\mu)=m_i(x)$ for $i=1,2,\cdots,s$. We say an ergodic measure $\mu$ is a \emph{hyperbolic measure} if all Lyapunov exponents are nonzero. 
    
    Ruelle's inequality connects metric entropy and the sum of the positive Lyapunov exponents, and it offers a useful tool in proving the existence of measures with some exponents different from zero. If the topological entropy of a diffeomorphism is not zero, then there is a measure with some of its exponents positive.  
		
\begin{Theorem}\label{RI}
    (Ruelle's inequality \cite[Theorem 2]{RD})Given a diffeomorphism $f\in$ {\em Diff}$^{1}(M)$ and a measure $\mu\in \mathcal{M}(f)$, let $\lambda_i(x),~i=1,2,\cdots,s(x)$, be the Lyapunov exponents of $f$ at the point $x$. Then $$h_{\mu}(f)\leq \int\lambda^{+}(x)d\mu,$$ where $\lambda^{+}(x)=\sum\limits_{\lambda_i(x)> 0} m_i(x)\lambda_i(x)$ and $m_i(x)$ is the multiplicity of $\lambda_i(x)$. Especially if the measure $\mu\in\mathcal{E}(f)$, then $$h_{\mu}(f)\leq \lambda^{+}(\mu),~\text{where $\lambda^{+}(\mu)=\sum\limits_{\lambda_i(\mu)> 0} m_i(\mu)\lambda_i(\mu)$}.$$
\end{Theorem}
	
\section{Horseshoe and Symbolic Dynamical System}
	
	Given a natural number $N\geq 2$, let $$\Sigma_{N}=\{\omega=(\cdots,\omega_{-1},\omega_0,\omega_1,\cdots): \omega_i \in\{0,1,\cdots,N-1\}, \text{ for $i\in\mathbb{Z}$} \}$$ be the space of two-sided sequences of $N$ symbols and $$\Sigma_N^1=\{\omega=(\omega_0,\omega_1,\cdots): \omega_i \in\{0,1,\cdots,N-1\}, \text{ for $i\in\mathbb{N}$} \}$$ the space of one-sided sequences of $N$ symbols. One can define a topology by noting that $\Sigma_N$ is the direct product of $\mathbb{Z}$ copies of the finite set $\{0,1,\cdots,N-1\}$, each with the discrete topoplogy, and using the product topology. The left shift $\sigma_N$ in $\Sigma_N$ is defined as: $$\sigma_{N}:\Sigma_N\rightarrow \Sigma_N,~~~~\sigma_N(\omega)=\omega'=(\cdots,\omega'_{-1},\omega'_0,\omega'_1,\cdots),$$ where $\omega'_i=\omega_{i+1}$ for every $i\in\mathbb{Z}$. The left shift $\sigma_N$ is a one-to-one map. Thus it is a homeomorphism of $\Sigma_N$. Sometimes the left shift $\sigma_N$ is also called a \emph{topological Bernoulli shift}.
	
	Let $M$ be a compact Riemannian manifold, $f:M\rightarrow M$ a $C^1$ diffeomorphism on $M$ and Diff$^1(M)$ the space of $C^1$ diffeomorphism on $M$. A subset $\Lambda\subset M$ is called an $f$-invariant set if $f(\Lambda)\subset\Lambda$. An $f$-invariant set $\Lambda$ is said to be a \emph{hyperbolic set} if there are two constants $C>0$, $\lambda\in(0,1)$ and an $Df$-invariant splitting $T_{\Lambda} M=E\bigoplus F$, such that $$\parallel Df^n\mid_{E(x)}\parallel\leq C\lambda^n,~~\parallel Df^{-n}\mid_{F(x)}\parallel\leq C\lambda^n,~\text{for any $x\in\Lambda$ and any $n\in\mathbb{N}$}.$$   
	
	Let $M_1,~M_2$ be the compact spaces, $f: M_1\rightarrow M_1$, $g: M_2\rightarrow M_2$ the homeomorphisms of compact spaces. The homeomorphism $f$ is called \emph{topologically conjugate} to the homeomorphism $g$, if there exists a homeomorphism $h: M_1\rightarrow M_2$ such that $ h\circ f=g\circ h$. 
	
\begin{Definition}\label{Horseshoe}
	Given a diffeomorphism $f\in$ {\em Diff}$^{1}(M)$, an $f$-invariant hyperbolic set $\Lambda$ is called a horseshoe of $f$, if the restriction of $f$ into $\Lambda$ is topologically conjugate to a subshift of finite type and the topological entropy $h_{top}(f|_{\Lambda})> 0$.
\end{Definition}

   From Peter Walters \cite[Theorem 4.28, Remark]{PW}, for a symbolic system $\Sigma_N$, the left shift $\sigma_N$ has invariant measures of arbitrary intermediate entropies. By \cite[Theorem 1.12]{PW}, one has that the left shift $\sigma_N$ is ergodic. Thus, for a symbolic system $\Sigma_N$, all possible entropies of ergodic measures form an interval $(0,\log N]$. The following Lemma \ref{SX} gives a similar result that all possible topological entropies of subshifts of $\Sigma_N$ also form an interval $ (0,\log N]$.
  
\begin{Lemma}\label{SX}(\cite[Page178-179]{PW})
    Given a natural number $N\geq 2$, for any real number $\alpha\in (0,\log N]$, there exists a compact $\sigma_N$-invariant set $\Sigma(\alpha)\subset\Sigma_N$ such that $$ h_{top}(\sigma_N|_{\Sigma(\alpha)})=\alpha.$$
\end{Lemma}
  
\begin{Lemma}\label{SH}
    Assume that a hyperbolic set $\Lambda\subset M$ is a horseshoe of $f$. Then for any $\varepsilon>0$ and $\alpha\in (0,h_{top}(f|_{\Lambda})]$, there exists an $f$-invariant compact set $\Lambda(\alpha)\subset \Lambda$ such that
  	
  \begin{enumerate}
  	\item[(1)] $\Lambda(\alpha)$ is a horseshoe of $f$;
  		
  	\item[(2)] $\lvert h_{top}(f|_{\Lambda(\alpha)})-\alpha\lvert <\varepsilon.$ 
  \end{enumerate} 
  	
\end{Lemma}

\begin{proof}
    Since $\Lambda\subset M$ is a horseshoe of $f$, by Definition \ref{Horseshoe}, there is a $N$ such that $f$ restricted on $\Lambda$ is topologically conjugate to a subshift of $N$ symbols. Consequently, there exists a homeomorphism $h:\Lambda\rightarrow \Sigma_N$ satisfying that $$ h\circ f|_{\Lambda}=\sigma_N\circ h,$$ where $ \sigma_N $ is a topological Bernoulli shift on $\Sigma_N$. Since topological entropy is a topological conjugate invariant, one has that $$ h_{top}(f|_{\Lambda})=h_{top}(\sigma_N)= \log N.$$ Then for any $\alpha\in (0,h_{top}(f|_{\Lambda})]$, according to Lemma \ref{SX}, there exists a compact $\sigma_N$-invariant set $\Sigma(\alpha)\subset\Sigma_N$ such that $$ h_{top}(\sigma_N|_{\Sigma(\alpha)})=\alpha.$$ 
    
    From Peter Walters \cite[Theorem 1.12 and Theorem 4.28, Remark]{PW}, there is an ergodic measure $\mu$ of left shift $\sigma_N$ such that $$h_{\mu}(\sigma_N)=h_{top}(\sigma_N|_{\Sigma(\alpha)}).$$ 
    
    Let $\nu=h_{*}\mu$, one can get that $\nu$ is a hyperbolic ergodic measure of $f$. Thus, $$h_{\nu}(f)=h_{\mu}(\sigma_N)=h_{top}(\sigma_N|_{\Sigma(\alpha)})=\alpha.$$ Then for any $\varepsilon>0$, there is a horseshoe $\Lambda(\alpha)\subset h^{-1}(\Sigma(\alpha)) \subset \Lambda $ of $f$ such that $$h_{\nu}(f)-\varepsilon < h_{top}(f|_{\Lambda(\alpha)}).$$ Therefore, $$\lvert h_{top}(f|_{\Lambda(\alpha)})-\alpha\lvert <\varepsilon.$$ 
\end{proof}

\section{Approximation Properties and Dominated Splitting}

   In this section, we introduce some approximation properties which provide fundamental ways to the heart of the proof of the Main Theorem. The first is due to Ma\~n\'e \cite[Theorem A]{MR}. 
       
\begin{Definition}
    A point $x\in M$ is called a strongly closable point of diffeomorphism $f\in$ {\rm Diff$^1(M)$}, if for any given $C^1$ neighborhood $\mathcal{U}$ of $f$ and $\varepsilon>0$, there is a diffeomorphism $g\in\mathcal{U}$ and a periodic point $y\in M$ of diffeomorphism $g$, such that 
    \begin{itemize} 
	  \item $g=f$ on $M\setminus\bigcup_{i\in\mathbb{Z}}B(f^i(x),\varepsilon)$;
		
	  \item $d(f^j(x),g^j(y))\leq\varepsilon$, for all $0\leq j\leq Per(y)$.	
	\end{itemize}	
\end{Definition}

    Let $\Sigma(f)$ be the set of strongly closable points of diffeomorphism $f$. Ma\~n\'e's ergodic closing lemma \cite[Theorem A]{MR} asserts that the set $\Sigma(f)$ has full measure w.r.t. any invariant probability measure.   

\begin{Lemma}\label{TECL}(Ma\~n\'e's ergodic closing lemma \cite[Theorem A]{MR})
	For any diffeomorphism $f\in$ {\rm Diff$^1(M)$} on a compact Riemannian manifold $M$, the set $\Sigma(f)$ has full probability measure for diffeomorphism $f$. 
\end{Lemma}

    For a Borel probability measure $\mu$, let $supp(\mu)$ be the support set of $\mu$. We say that a point $x\in supp(\mu)$, if for any $r>0$, the non-empty open set $B(x,r)$ has positive measure. The following Proposition \ref{SMC} certifies the relationship between $supp(\mu)$ and those of $\{\mu_n\}$ which approaches to the measure $\mu$.  

\begin{Proposition}\label{SMC}
	Let $\{\mu_n\}$ be a sequence of Borel probability measure on a compact metric space $M$. If $\mu_n\rightarrow \mu$ in the weak$^*$ topology sense, then one has that $$supp(\mu)\subset \bigcap_{m\geq 1}\overline{\bigcup_{n\geq m}supp(\mu_n)}.$$
\end{Proposition}

\begin{proof}
	By the definition of the support set of a measure, for every $x\in supp(\mu)$ and every $k\in \mathbb{N}$, there exists $\delta>0$ such that $$\mu(B(x,\frac{1}{k}))>\delta>0,$$ where $B(x,\frac{1}{k})$ is an open ball whose centre is $x$ and radius is $\frac{1}{k}$. Since $\mu_n\rightarrow \mu$, by the property that the limit inferior of $\{ \mu_n(E)\}$ is not less than $\mu(E)$ for any open set $E$ when $\mu_n\rightarrow \mu$, one has that $$\liminf_{n\rightarrow \infty} \mu_{n}(B(x,\frac{1}{k})) \geq \mu(B(x,\frac{1}{k})).$$  It means that $\liminf\limits_{n\rightarrow \infty} \mu_{n}(B(x,\frac{1}{k}))>\delta>0.$ Therefore, there exists a constant $N\in\mathbb{N}^+$ such that 
	$$\mu_{n}(B(x,\frac{1}{k}))>\frac{\delta}{2}>0,~{\rm for~ any}~n\geq N.$$ Consequently, there is a sequence $\{y_n\triangleq y(k,n)\}$ satisfying that $$y(k,n)\in supp(\mu_{n})\bigcap B(x,\frac{1}{k}),~{\rm for~ any}~n\geq N.$$ Let $k\rightarrow \infty$, one can choose a subsequence $\{y(k,n_k)\}$ such that $n_{k}\rightarrow \infty$ and $ y(k,n_{k})\rightarrow x$. Thus, for any $m\in \mathbb{N}$, one has that $$x\in \overline{\bigcup_{n\geq m}supp(\mu_n)}.$$ Then, $x\in \bigcap_{m\geq 1}\overline{\bigcup_{n\geq m}supp(\mu_n)}$. Thus, the proposition is proved.
\end{proof}

    Lyapunov exponents play a key role in understanding the ergodic behavior of a dynamical system. Given a measure $\mu\in \mathcal{E}(f)$ on a $d$-dimensional compact manifold, by Oseledec's theorem \cite[Theorem 3]{OV}, one can get $d$-Lyapunov exponents $\lambda_1(\mu)\leq\lambda_2(\mu)\leq\cdots\leq\lambda_d(\mu)$. Denoted by $L(\mu)\in \mathbb{R}^d$ the $d$-dimensional vector of the Lyapunov exponents of $\mu$, with multiplicity, endowed with an increasing order.
     
    As we know, periodic points play an important role in the study of the dynamics of diffeomorphisms. For a periodic orbit, one also has the $d$-dimensional vector of Lyapunov exponents. Abdenur, Bonatti and Crovisier \cite[Proposition 6.1]{ABC} proved that the Lyapunov exponents of an ergodic invariant measure can be approximated by those of periodic orbits.
  
\begin{Theorem}\label{NHGD}(\cite[Proposition 6.1]{ABC})
    Let $\mu$ be an ergodic invariant measure of a diffeomorphism $f$ on a compact manifold $M$. Fix a $C^1$ neighborhood $\mathcal{U}$ of $f$, a neighborhood $\mathcal{V}$ of $\mu$ in the space of probability measures with the weak topology and a neighborhood $\mathcal{O}$ of $L(\mu)$ in $\mathbb{R}^d$, there is $g\in \mathcal{U}$ and a periodic orbit $\gamma$ of $g$ such that the Dirac measure $\mu_{\gamma}$ associated to $\gamma$ belongs to $\mathcal{V}$, and its Lyapunov vector $L(\mu_{\gamma})$ belongs to $\mathcal{O}$.
\end{Theorem}

   Given an $f$-invariant compact set $\Lambda\subset M$, we say that there is a \emph{dominated splitting} over $\Lambda$ if there exist two constants $C>0$, $\lambda\in(0,1)$ and an $Df$-invariant splitting $T_\Lambda M=E\bigoplus F$ such that $$\frac{\parallel Df^n\mid_{E(x)}\parallel}{m(Df^n\mid_{F(x)})}\leqslant C\lambda^{n},~\text{for any $x\in\Lambda$ and any $n\in\mathbb{N}$}.$$  Especially if dim$F = 1$, then one has that 
   $$\frac{\parallel Df^n\mid_{E(x)}\parallel}{\parallel Df^n\mid_{F(x)}\parallel}=\frac{\parallel Df^n\mid_{E(x)}\parallel}{m(Df^n\mid_{F(x)})}\leqslant C\lambda^{n},~\text{for any $x\in\Lambda$ and any $n\in\mathbb{N}$}.$$
   
\begin{Definition}\label{NDS}
   Let $\{f_n\}$ be a sequence of diffeomorphisms of a compact manifold $M$, $\{p_n: p_n \text{ is a periodic point of } f_n\}$ a sequence of points. We say that there is a uniformly dominated splitting over $\mathcal{O}rb(f_n, p_n)$ for all $n$, if there exist two uniform constants $C>0$ and $\lambda\in(0,1)$ such that for any $n\in\mathbb{N}$ and an $Df_n$-invariant splitting $T_{\mathcal{O}rb(f_n, p_n)}M=E\bigoplus F$, one has that
   $$\frac{\parallel Df_n^k\mid_{E(x)}\parallel}{m(Df_n^k\mid_{F(x)})}\leqslant C\lambda^{k},~\text{for any $x\in \mathcal{O}rb(f_n, p_n)$ and any $k\in\mathbb{N}^+$}.$$
\end{Definition}
   
   In \cite[Theorem 11]{KA1}, Katok established a well-known result about the relationship between metric entropy and horseshoe in quantity for $C^{1+\alpha}(\alpha>0)$ diffeomorphisms. Now, we consider a $C^1$ diffeomorphism preserving a hyperbolic ergodic measure with positive measure entropy. The following Theorem \ref{HHS} given by Gelfert says that if the support set admits a dominated splitting, then one can get a similar result about the quantitative relation between metric entropy and horseshoe. 
   
\begin{Theorem}\label{HHS}(\cite[Theorem 1]{GEL})
   Let $\mu$ be an ergodic measure of a $C^1$ diffeomorphism $f$ with $h_{\mu}(f)>0$ on a compact Riemannian surface $M$. If there is a dominated splitting over $supp(\mu)$, then for any $\varepsilon>0$, there exists a hyperbolic horseshoe $\Lambda$ such that 
   $$h_{\mu}(f)-\varepsilon< h_{top}(f|_{\Lambda}),$$ where $\Lambda$ is contained in a $\varepsilon$-neighborhood of $supp(\mu)$.  
\end{Theorem}

   A classical construction due to Newhouse \cite{NW} creates horseshoes from hyperbolic periodic orbits with large period and weak domination through local $C^1$-perturbations. When there is no dominated splitting, the following Theorem \ref{EDWDS}  given by Buzzi, Crovisier and Fisher says that there is a horseshoe with entropy arbitrarily close to an upper bound following from Ruelle's inequality (Theorem \ref{RI}). Let us consider the case where the dimension of the manifold is two.

\begin{Theorem}\label{EDWDS} (\cite[Theorem 4.1]{BCF})
   For any $C^{1}$ diffeomorphism $f$ of a compact two-dimensional manifold $M$ and any $C^1$ neighborhood $\mathcal{U}$ of $f$, there exists $T\geq1$ with the following property. If $p$ is a periodic point of $f$ with period at least $T$ and whose orbit has no dominated splitting, then there exists $g\in\mathcal{U}$ containing a horseshoe $\Lambda$ such that $$h_{top}(g|_{\Lambda})\geq min\{\lambda^{+}(p),\lambda^{-}(p)\},$$ where $\lambda^{+}(p)$ and $-\lambda^{-}(p)$ are the non-negative and non-positive Lyapunov exponents of $f$ at the point $p$ respectively. 
\end{Theorem}

\section{Proof of the Main Theorem}
       
   In this section, we give the proof of the Main Theorem. Hereafter, we always assume that dim$M=2$. Let $\delta_x$ be the atomic measure at $x$. 
    
\begin{proof}[Proof of the Main Theorem]
   For the $C^1$ diffeomorphism $f$ on the compact Riemannian surface $M$ with positive topological entropy, by the variational principle (Theorem \ref{VP}), $$h_{top}(f)=\sup\{h_{\mu}(f): \mu\in\mathcal{E}(f)\},$$ one can find an ergodic measure $\mu$ such that $h_{\mu}(f)>0$. Since $ h_{\mu}(f)=h_{\mu}(f^{-1}) $ and Ruelle's inequality(Theorem \ref{RI}), one has that $$h_{\mu}(f)\leq min \{\lambda^{+}(\mu), \lambda^{-}(\mu)\},$$ where $\lambda^{+}(\mu)$ and $-\lambda^{-}(\mu)$ are the positive exponent and negative exponent of $ f $ respectively.
   Noting that dim$M=2$, then all the Lyapunov exponents of $f$ w.r.t. the measure $\mu$ are different from zero. Therefore, the ergodic invariant measure $\mu$ is hyperbolic.

   Since $\mu$ is ergodic, one can obtain that the set $A\triangleq\{x\in M: \lim\limits_{n\rightarrow \infty}\frac{1}{n} \sum\limits_{i=0}^{n-1}\delta_{f^{i}x}=\mu\}$ is a full measurable set. According to Ma\~n\'e's ergodic closing lemma (Lemma \ref{TECL}), one has that $$\mu(A\cap\Sigma(f))=1.$$ It means that $\mu$-alomost every $x\in A\cap\Sigma(f)$ is a strongly closable point. Choose a point $p\in A\cap\Sigma(f)$, taking $\varepsilon=\dfrac{1}{n}$ successively, by Theorem \ref{NHGD}, there is a sequence of diffeomorphisms $\{f_n\}$ and a sequence of points $\{p_n: p_n \text{ is the periodic point of } f_n \}$, such that $$f(y)=f_{n}(y), {\rm \ for \ every \ } y \in M\setminus\bigcup_{i\in\mathbb{Z}}B(f^i(p),\frac{1}{n}), {\rm \ and }$$ $$d(f^j(p),f_{n}^j(p_{n}))\leq \frac{1}{n}, ~{\rm for\ all \ } 0\leq j\leq \pi(p_n).$$ According to Theorem \ref{NHGD}, one also has that $$\lim_{n\rightarrow\infty}\lambda^{+}(p_n) = \lambda^{+}(\mu) {\rm \ \ \ and \ \ \ } \lim_{n\rightarrow\infty}\lambda^{-}(p_n)\rightarrow \lambda^{-}(\mu),$$ where $\lambda^{+}(p_n)$ and $-\lambda^{-}(p_n)$ are the positive and negative Lyapunov exponents of $f_n$ at the point $p_n$ respectively. We claim that one can choose the point $p$ which is nonperiodic. Otherwise, it would contradict that $f$ has positive metric entropy. Since $f_n\rightarrow f$, $p_n\rightarrow p$ and the continuity of diffeomorphsm $f$, $$\pi(p_n)\rightarrow \infty {\rm \ as \ }n\rightarrow\infty.$$

\begin{claim}
   Choose the nonperiodic point $p$, for the sequence of diffeomorphisms $\{f_{n}\}$ and the sequence of points $\{p_{n}\}$, one has that %\item[(1)]$\pi(p_n)\rightarrow \infty$ as $n \rightarrow \infty$;
   $$\mu_{n}=\frac{1}{\pi(p_n)}\sum\limits^{\pi(p_n)-1}_{i=0}\delta_{f_{n}^{i}(p_{n})}\rightarrow \mu \text{ as $n \rightarrow \infty$}.$$	
\end{claim}   
   
\begin{proof}(Proof of claim)
   By Birkhoff ergodic theorem, for $\mu$-almost every $x\in M$ and any continuous function $g:M\rightarrow\mathbb{R}$, one has that $$\lim_{j\rightarrow \infty}\frac{1}{j}\sum^{j-1}_{i=0}g(f^{i}(x))= \int gd\mu.$$ Then, for any $\varepsilon>0$, there exists an integer $N_{1}>0$ such that $$|\frac{1}{j}\sum^{j-1}_{i=0}g(f^{i}(p))-\int gd\mu|<\frac{\varepsilon}{2},~~\text{for any}~j\geq N_{1}.$$ Since $\pi(p_n)\rightarrow \infty$, there exists an integer $N_{2}>0$ such that $\pi(p_{n})> N_{1}$ for any $n\geq N_{2}$. Therefore, $$\lvert\frac{1}{\pi(p_{n})}\sum^{\pi(p_{n})-1}_{i=0}g(f^{i}(p))-\int gd\mu\rvert<\frac{\varepsilon}{2},~~\text{for any}~n\geq N_{2}.$$ Since $f_n\rightarrow f$, $p_n\rightarrow p$ and $g$ is a continuous function, there exists an integer $N_{3}>0$ such that $$\lvert\frac{1}{\pi(p_{n})}\sum^{\pi(p_{n})-1}_{i=0}g(f^{i}_{n}(p_{n}))-\frac{1}{\pi(p_{n})}\sum^{\pi(p_{n})-1}_{i=0}g(f^{i}(p))\rvert<\frac{\varepsilon}{2},~~\text{for any }n\geq N_{3}.$$ 
   Therefore, taking an integer $N=max \{N_{2}, N_{3}\}$, for any $n\geq N$, one has that
\begin{eqnarray*}
	\lvert\int gd\mu_{n}-\int gd\mu\rvert &=& \lvert\frac{1}{\pi(p_{n})}\sum^{\pi(p_{n})-1}_{i=0}\int gd\delta_{f_n^i(p_n)}-\int gd\mu\rvert =\lvert\frac{1}{\pi(p_{n})}\sum^{\pi(p_{n})-1}_{i=0}g(f^{i}_{n}(p_{n}))- \int gd\mu\rvert\\ &\leq& \lvert\frac{1}{\pi(p_{n})}\sum^{\pi(p_{n})-1}_{i=0}g(f^{i}_{n}(p_{n}))-\frac{1}{\pi(p_{n})}\sum^{\pi(p_{n})-1}_{i=0}g(f^{i}(p))\rvert \\&+& \lvert\frac{1}{\pi(p_{n})}\sum^{\pi(p_{n})-1}_{i=0}g(f^{i}(p))-\int gd\mu\rvert<\frac{\varepsilon}{2}+\frac{\varepsilon}{2}=\varepsilon.
\end{eqnarray*}   
   Thus, $\mu_n\rightarrow \mu$ as $n \rightarrow \infty$ in the weak$^*$ topology sense.	
\end{proof}  
    
  Now we consider the following two different cases:

  {\textbf{Case 1:}} There is no uniformly dominated splitting over $\mathcal{O}rb(f_n, p_n)$ for all $n$, then one can take a subsequence $\{p_{n_k}\}$ of $\{p_n\}$ such that there are no dominated splittings over $\mathcal{O}rb(f_n, p_n)$ for all $k$. For discussion purposes, we assume that there exists no dominated splittings over $\mathcal{O}rb(f_n, p_n)$ for all $n$.
  
  By the variational principle (Theorem \ref{VP}), for any $\varepsilon>0$ there is an ergodic hyperbolic measure $\mu$ such that $$h_{\mu}(f)> h_{top}(f)-\varepsilon.$$ Fixed the neighborhood $\mathcal{U}$ of $f$, for $p_n\in M$ and any $C^{1}$ neighborhood $\mathcal{U}_n\subset\mathcal{U}$ of $f_n$, by Theorem \ref{EDWDS}, there exists $g\in\mathcal{U}_n$ containing a horseshoe $\Lambda$ such that $$h_{top}(g|_{\Lambda})\geq min \{\lambda^{+}(p_n), \lambda^{-}(p_n)\}.$$ Since $\mathcal{U}_n\in\mathcal{U}$, one has that $$g\in\mathcal{U}_n\subset\mathcal{U}.$$ As $\lambda^{+}(p_n)\rightarrow \lambda^{+}(\mu)$ and $\lambda^{-}(p_n)\rightarrow \lambda^{-}(\mu)$ when $n \rightarrow \infty$, one can take $n$ large enough such that $h_{top}(g|_{\Lambda})$ can arbitrarily close $min \{\lambda^{+}(\mu), \lambda^{-}(\mu)\}$. Since $h_{\mu}(f)\leq min \{\lambda^{+}(\mu), \lambda^{-}(\mu)\}$, one has that  $$h_{top}(g|_{\Lambda}) \geq h_{\mu}(f)>0.$$ Consequently, $$h_{top}(g|_{\Lambda}) \geq h_{\mu}(f)> h_{top}(f)-\varepsilon.$$
  
  For proving our goal, we discuss two different conditions.
  
  $(\uppercase\expandafter{\romannumeral1})$ If for any $\varepsilon>0$,  one has that $h_{top}(g|_{\Lambda}) <h_{\mu}(f)+\varepsilon$, then 
  $$h_{top}(g|_{\Lambda}) <h_{top}(f)+\varepsilon.$$ Therefore, the result follows, i.e. there exists a diffeomorphism $g\in\mathcal{U}$ exhibiting a horseshoe $\Lambda$ such that $$\lvert h_{top}(g|_{\Lambda})-h_{top}(f)\rvert<\varepsilon,~\text{for any $\varepsilon>0$}.$$ 
  
  $(\uppercase\expandafter{\romannumeral2})$ If there exists a real number $c>0$ such that $h_{\mu}(f)+c < h_{top}(g|_{\Lambda})$, then $$ h_{top}(g|_{\Lambda})>h_{top}(f)+c-\varepsilon \triangleq h_{top}(f)+c_{0}, $$ where $ c_{0}=c-\varepsilon $ is a positive number. Thus, $$h_{top}(f)\in (0, h_{top}(g|_{\Lambda})].$$ By Lemma \ref{SH}, there exists a horseshoe $\Lambda^{\ast}\triangleq \Lambda(h_{top}(f))\subset \Lambda$ such that $$\lvert h_{top}(g|_{\Lambda^{\ast}})-h_{top}(f)\lvert <\varepsilon,~\text{for any $\varepsilon>0$}.$$ Consequently, for any $\varepsilon>0$ there exists $g\in\mathcal{U}$ exhibiting a horseshoe $\Lambda^{\ast}$ such that 
  $$\lvert h_{top}(g|_{\Lambda^{\ast}})-h_{top}(f)\rvert<\varepsilon.$$ 
  
  {\textbf{Case 2.}} There exists a uniformly dominated splitting over $\mathcal{O}rb(f_n, p_n)$ for all $n$.

\begin{claim}
   The support set $supp(\mu)$ admits a dominated splitting.
\end{claim}

\begin{proof}(Proof of claim)
   By definition, the support set $supp(\mu)\subset M$ is a closed invariant set of $f$. Since $\mu_n\rightarrow \mu$ in the weak$^*$ topology sense, by Proposition \ref{SMC}, one obtains that $$supp(\mu)\subset \bigcap_{m\geq 1}\overline{\bigcup_{n\geq m}supp(\mu_n)}.$$ Therefore, for each $x\in supp(\mu)$, one has that $$x\in \overline{\bigcup_{n\geq m}supp(\mu_n)}, {\rm \ \ for \ any \ integer \ } m\geq 1.$$ Then, there exists an integer $m_1>0$ such that $x_{m_1}\in supp(\mu_{m_1})\bigcap B(x,1)$. Similarly, there is an integer $m_2 > m_1$ such that $x_{m_2}\in supp(\mu_{m_2})\bigcap B(x,\frac{1}{2})$. Successively, there exists an integer $m_k > m_{k-1}$ such that $$ x_{m_k}\in supp(\mu_{m_k})\bigcap B(x,\frac{1}{k}).$$ Therefore, there is a sequence of points $\{x_{m_k}: x_{m_k} {\rm \ is \ the \ periodic \ point \ of \ } f_{m_k}\}$ satisfying that $$x_{m_k}\rightarrow x,~~~f_{m_k}\rightarrow f.$$ 
    
   Since there exists a uniformly dominated splitting over $\mathcal{O}rb(f_n, p_n)$ for all $n$ and dim$M=2$, by Definition \ref{NDS}, there exist two uniform constants $C>0$ and $\lambda\in(0,1)$ such that for any $n\in\mathbb{N}^+$ and an $Df_{m_k}$-invariant splitting $T_{\mathcal{O}rb(f_{m_k}, p_{m_k})}M=E\bigoplus F$, one has that $$\frac{\parallel Df_{m_k}^n\mid_{E(x)}\parallel}{\parallel Df_{m_k}^n\mid_{F(x)}\parallel}\leqslant C\lambda^n,~\text{for any } x\in \mathcal{O}rb(f_{m_k}, p_{m_k}).$$ Denote the basis of $E(x_{m_k})$ by $\{i^{1}_{x_{m_k}}\}$, the basis of $F(x_{m_k})$ by $\{j^{1}_{x_{m_k}}\}$. Since $x_{m_k}\rightarrow x,~f_{m_k}\rightarrow f$ and the continuity of diffeomorphisms, one has that
   $$i^{1}_{x_{m_k}}\rightarrow i^{1}_{x},~~~j^{1}_{x_{m_k}}\rightarrow j^{1}_{x},$$ where $\{i^{1}_{x}\}$ and $\{j^{1}_{x}\}$ are the bases of $E(x)$ and $F(x)$ respectively. Thus $$E(x_{m_k})\rightarrow E(x),~F(x_{m_k})\rightarrow F(x),~T_{x_{m_k}}M\rightarrow T_{x}M.$$ Therefore, one can obtain a decomposition $$T_{x}M=E(x)\bigoplus F(x).$$
   
   Since $$Df_{m_k}(E(x_{m_k}))=E(f_{m_k}(x_{m_k}))~\text{and}~Df_{m_k}(F(x_{m_k}))=F(f_{m_k}(x_{m_k})),$$ one has that $$Df(E(x))=E(f(x))~\text{and}~Df(F(x))=F(f(x)).$$ Thus, the splitting $T_{x}M=E(x)\bigoplus F(x)$ is an $Df$-invariant splitting. By the continuity of diffeomorphism $f$, one has that $$\frac{\parallel Df^n\mid_{E(x)}\parallel}{\parallel Df^n\mid_{F(x)}\parallel}\leqslant C\lambda^{n}, ~\text{for any $n\in\mathbb{N}$}.$$ Since the abitrariness of $x\in supp(\mu)$, there exists a dominated splitting over $supp(\mu)$. Therefore, the claim is proved.
\end{proof}   

   For any $\varepsilon>0$, according to the variational principle (Theorem \ref{VP}) there exists an ergodic hyperbolic measure $\mu$ such that $$h_{\mu}(f)> h_{top}(f)-\frac{\varepsilon}{2}.$$ For this hyperbolic measure $\mu\in\mathcal{E}(f)$, by Theorem \ref{HHS}, there exists a horseshoe $\Lambda$ satisfying that $$h_{\mu}(f)-\frac{\varepsilon}{2}< h_{top}(f|_{\Lambda}).$$ Consequently, $$h_{top}(f)<h_{\mu}(f)+\frac{\varepsilon}{2}<h_{top}(f|_{\Lambda})+\varepsilon.$$ 
   
   By \cite[Definition 7.6, Remarks(12)]{PW}, one has that $$h_{top}(f)\geq h_{top}(f|_{\Lambda}).$$ Applying to the variational principle (Theorem \ref{VP}), one obtains that $$h_{top}(f)\geq h_{\mu}(f)>h_{top}(f)-\frac{\varepsilon}{2}.$$ Therefore, $$h_{top}(f|_{\Lambda})-\varepsilon< h_{top}(f)-\frac{\varepsilon}{2}<h_{top}(f)<h_{top}(f|_{\Lambda})+\varepsilon.$$ Thus, $$\lvert h_{top}(f|_{\Lambda})-h_{top}(f)\rvert<\varepsilon.$$ Taking $g=f$, there exists a diffeomorphism $g\in\mathcal{U}$ exhibiting a horseshoe $\Lambda$ such that $$\lvert h_{top}(g|_{\Lambda})-h_{top}(f)\rvert<\varepsilon,~\text{for any $\varepsilon>0$}.$$ This completes the proof of the Main Theorem.
\end{proof}
 
\noindent\textbf{Acknowledgements.} We would like to thank Dawei Yang for his useful suggestions. We would also like to thank Gang Liao and Lingmin Liao for their useful conversations.

\addcontentsline{toc}{section}{Reference}
\bibliographystyle{amsplain}

\begin{thebibliography}{99}  
	\bibitem{ABC} Flavio Abdenur, Christian Bonatti, and Sylvain Crovisier, {\em Nonuniform hyperbolicity for $ C^{1} $-generic diffeomorphisms}, Israel J. Math.\textbf{183} (2011), 1\textendash60. 
	
	\bibitem{AKM}Roy Adler, Alan Konheim, and M. McAndrew, {\em Topological entropy}, Trans. Amer. Math.
	Soc. \textbf{114} (1965), 309\textendash319. 
	
	\bibitem{BR}Rufus Bowen, {\em Entropy for group endomorphisms and homogeneous spaces}, Trans. Amer.
	Math. Soc. \textbf{153} (1971), 401\textendash414.
	
	\bibitem{BCF}Jerome Buzzi, Sylvain Crovisier, and Todd Fisher, {\em Entropy of $ C^{1} $ diffeomorphisms without a dominated splitting}, arXiv:1606.01765.   
	
	\bibitem{DEI}Efim Dinaburg, {\em A correlation between topological entropy and metric entropy}, Dokl. Akad. Nauk SSSR \textbf{190} (1970), 19\textendash22.
	
	\bibitem{Gan}Shaobo Gan, {\em Horseshoe and entropy for $ C^{1} $ surface diffeomorphisms}, Nonlinearity \textbf{15} (2002), no. 3, 841\textendash848.
	
	\bibitem{GEL}Katrin Gelfert, {\em Horseshoes for diffeomorphisms preserving hyperbolic measures}, Math. Z. \textbf{283} (2016), no. 3-4, 685\textendash701.
	
	\bibitem{Ka}Anatole Katok, {\em Lyapunov exponents, entropy and periodic orbits for diffeomorphisms}, Inst. Hautes \'{E}tudes Sci. Publ. Math. (1980), no. 51, 137\textendash173.   
	
	\bibitem{KA1}Anatole Katok, {\em Nonuniform hyperbolicity and structure of smooth dynamical systems}, Proceedings of the {I}nternational {C}ongress of {M}athematicians, {V}ol.\ 1, 2 ({W}arsaw, 1983), PWN, Warsaw, 1984, pp. 1245\textendash1253. 
	
	\bibitem{K1}Andre\v{\i} Kolmogorov, {\em A new metric invariant of transient dynamical systems and automorphisms in Lebesgue spaces}, Dokl. Akad. Nauk SSSR (N.S.) \textbf{119} (1958), 861\textendash864. 
	
	\bibitem{MR}Ricardo Ma\~{n}\'{e}, {\em  An ergodic closing lemma}, Ann. of Math. (2) \textbf{116} (1982), no. 3, 503\textendash540.
	
	\bibitem{NW}Sheldon Newhouse, {\em Topological entropy and {H}ausdorff dimension for area
		preserving diffeomorphisms of surfaces}, (1978), 323-334. Ast\'{e}risque, No. 51.  
	
	\bibitem{OV}Valery Oseledec, {\em A multiplicative ergodic theorem.characteristic ljapunov, exponents of dynamical systems}, Trudy Moskov. Mat. Ob\v{s}\v{C}. \textbf{19} (1968), 179\textendash210.  
	
	\bibitem{RD}David Ruelle, {\em An inequality for the entropy of differentiable maps}, Bol. Soc. Brasil. Mat. \textbf{9} (1978), no. 1, 83\textendash87.
	
	\bibitem{S}Yakov Sina\u{\i}, {\em On the concept of entropy for a dynamic system}, Dokl. Akad. Nauk SSSR \textbf{124} (1959), 768\textendash771.   
	
	\bibitem{SS}Stephen Smale, {\em A structurally stable differentiable homeomorphism with an infinite number of periodic points}, Qualitative methods in the theory of non-linear vibrations ({P}roc. {I}nternat. {S}ympos. {N}on-linear {V}ibrations, {V}ol. {II}, 1961), Izdat. Akad. Nauk Ukrain. SSR, Kiev, 1963, pp. 365\textendash366. 
	
	\bibitem{PW}Peter Walters, {\em An introduction to ergodic theory}, Graduate Texts in Mathematics, vol. 79, Springer-Verlag, New York-Berlin, 1982.      
	
	
\end{thebibliography}

\noindent\text{Wanlou Wu}\\
School of Mathematical Sciences\\
Soochow University, Suzhou, 215006, P.R. China\\
wuwanlou@163.com, wanlouwu1989@gmail.com\\

\noindent\text{Jiansong Liu}\\
School of Mathematical Sciences\\	
Soochow University, Suzhou, 215006, P.R. China\\
jsliu1205@163.com, jsliu@stu.suda.edu.cn\\

\end{document}